\documentclass[a4paper,12pt]{article}     
\usepackage{amsmath,amscd,amssymb}        
\usepackage{latexsym}                     
\usepackage[english, german]{babel}                

\usepackage{theorem}                 

\newtheorem{theorem}{Theorem}
\newtheorem{corollary}{Corollary}
\newtheorem{proposition}{Proposition}
\newtheorem{lemma}{Lemma}

\theorembodyfont{\rmfamily}
\newtheorem{definition}{Definition}
\newtheorem{remark}{Remark}

\newenvironment{proof}{\begin{ProofwCaption}{Proof}}{\end{ProofwCaption}}
\newenvironment{proof*}[1]{\begin{ProofwCaption}{{#1}}}{\end{ProofwCaption}}
\newenvironment{ProofwCaption}[1]%
  {\addvspace\theorempreskipamount \noindent{\it #1.}\rm}%
  {\qed \par \addvspace\theorempostskipamount}
\newcommand{\qedsymbol}{{\rm $\Box$}}
\newcommand{\qed}{\hfill\qedsymbol}

\newcommand{\CC}{{\mathbb C}}

\newcommand{\RR}{{\mathbb R}}
\newcommand{\ZZ}{{\mathbb Z}}

\newcommand{\calG}{{\mathcal G}}
\newcommand{\calB}{{\mathcal B}}

\title{Enhanced equivariant Saito duality}
\author{Wolfgang Ebeling and Sabir M.~Gusein-Zade
\thanks{Partially supported by DFG (Mercator fellowship, Eb 102/8-1), RFBR-13-01-00755
and NSh-5138.2014.1.
AMS 2010 Math. Subject Classification: 14R20, 19A22, 14J33, 57R18, 58K10.
}
}
\date{}

\begin{document}
\selectlanguage{english}

\maketitle

\begin{abstract}
In a previous paper, the authors defined an equivariant version of the so-called Saito duality between the monodromy zeta functions
as a sort of Fourier transform between the Burnside rings of
an abelian group and of its group of characters. 
Here a so-called enhanced Burnside ring $\widehat{B}(G)$ of a finite group $G$ is defined.
An element of it is represented by a finite $G$-set with a $G$-equivariant transformation
and with characters of the isotropy subgroups associated to all points.
One gives an enhanced version of the equivariant Saito duality. For a complex analytic $G$-manifold
with a $G$-equivariant transformation of it one has an enhanced equivariant Euler characteristic
with values in a completion of $\widehat{B}(G)$. It is proved that the (reduced) enhanced equivariant
Euler characteristics of the Milnor fibres of Berglund-H\"ubsch dual invertible polynomials
coincide up to sign and show that this implies the result about orbifold zeta functions of
Berglund-H\"ubsch-Henningson dual pairs obtained earlier.
\end{abstract}

\section*{Introduction}
In \cite{Ar} V.~Arnold has observed a so-called strange duality between the 14 exceptional unimodular
singularities. K.~Saito has described a duality between the monodromy zeta functions of
dual singularities \cite{Saito1, Saito2}. The Arnold strange duality can be considered as
a special case of the Berglund-H\"ubsch mirror symmetry \cite{BH1}. This symmetry includes
a duality on the set of the so-called invertible polynomials: quasihomogeneous polynomials with
the number of monomials equal to the number of variables. In \cite{BLMS} we gave an equivariant
version of the Saito duality as a sort of Fourier transform between the Burnside rings of
an abelian group and of its group of characters. The group of diagonal symmetries of an invertible
polynomial is canonically isomorphic to the group of characters of the group of diagonal
symmetries of the dual polynomial. The equivariant Saito duality is a generalization of the usual
Saito duality to all invertible polynomials.

The Berglund-H\"ubsch duality has an orbifold version \cite{BH2} which considers a pair $(f,G)$
consisting of an invertible polynomial and a subgroup $G$ of the group of diagonal symmetries
of $f$ together with a dual pair $(\widetilde{f},\widetilde{G})$. The notion of the orbifold
Euler characteristic of a $G$-set was introduced in \cite{DHVW1, DHVW2}. In \cite{MMJ} we have
shown that the (reduced) orbifold Euler characteristic of the Milnor fibres of
Berglund-H\"ubsch-Henningson dual pairs coincide up to sign. The notion of the orbifold monodromy zeta
function was essentially defined in \cite[Definition 5.10]{ET}. In \cite{PEMS} it was shown
that the (reduced) orbifold zeta functions of Berglund-H\"ubsch-Henningson dual pairs either
coincide or are inverse to each other (depending on the number of variables).

Here we define a so-called enhanced Burnside ring $\widehat{B}(G)$ of a finite group $G$
and give an enhanced version of the equivariant Saito duality. For a complex analytic $G$-manifold
with a $G$-equivariant transformation of it one has an enhanced equivariant Euler characteristic
with values in a completion of $\widehat{B}(G)$. We prove that the (reduced) enhanced equivariant
Euler characteristics of the Milnor fibres of Berglund-H\"ubsch dual invertible polynomials
coincide up to sign and show that this implies the result about orbifold zeta functions of
Berglund-H\"ubsch-Henningson dual pairs from \cite{PEMS}. In fact a similar construction for
the equivariant Saito duality from \cite{BLMS} permits to deduce the result of \cite{MMJ}
from the main result of \cite{BLMS} directly.

\section{Enhanced Burnside ring}
Let $G$ be a finite group.

\begin{definition} \label{def1}
A {\em finite enhanced} $G$-{\em set} is a triple $(X, h, \alpha)$, where:
\begin{enumerate}
 \item[1)] $X$ is a finite $G$-set;
 \item[2)] $h$ is a one-to-one $G$-equivariant map $X\to X$;
 \item[3)] $\alpha$ associates to each point $x\in X$ a one-dimensional (complex) representation
 $\alpha_x$ of the isotropy subgroup $G_x=\{a\in G: ax=a\}$ of the point $x$ so that:
 \newline (a) for $a\in G$ one has $\alpha_{ax}(b)=\alpha_x(a^{-1}ba)$, where $b\in G_{ax}=aG_xa^{-1}$;
 \newline (b) $\alpha_{h(x)}(b)=\alpha_x(b)$.
\end{enumerate}

\end{definition}

An isomorphism between two finite enhanced $G$-sets $(X, h, \alpha)$ and $(X', h', \alpha')$
is defined in the natural way, namely it is a $G$-equivariant map $\varphi:X\to X'$ which satisfies
the conditions $\varphi\circ h=h'\circ\varphi$, $\alpha_{\varphi(x)}'=\alpha_x$.
The (Cartesian) product of two enhanced $G$-sets $(X', h', \alpha')$ and $(X'', h'', \alpha'')$
is the enhanced $G$-set $(X, h, \alpha)$ such that $X=X'\times X''$, $h(x',x'')=(h'(x'),h''(x''))$, and
for $b\in G_{(x',x'')}=G_{x'}\cap G_{x''}$ one has
$\alpha_{(x',x'')}(b)=\alpha_{x'}'(b)\cdot\alpha_{x''}''(b)$.

\begin{definition} \label{def2}
The {\em enhanced Burnside ring} $\widehat{B}(G)$ of the group $G$ is the Gro\-then\-dieck group
of finite enhanced $G$-sets with respect to the disjoint union.
The multiplication in $\widehat{B}(G)$ is defined by the Cartesian product of enhanced $G$-sets.
\end{definition}

In \cite{RMC}, there was considered the Grothendieck ring $\widetilde{B}(G)$ of so-called
finite equipped $G$-sets. One can say that a finite equipped $G$-set is an enhanced $G$-set
with the map $h$ missing. There is a natural homomorphism from the enhanced Burnside ring
$\widehat{B}(G)$ to $\widetilde{B}(G)$.

\begin{remark}
 One can define a pre-$\lambda$-ring structure (see, e.g., \cite{Knutson}) on the enhanced Burnside ring
 $\widehat{B}(G)$. Namely, one defines the symmetric power $S^n\widehat{X}$ of a finite enhanced
 $G$-set $\widehat{X}=(X,h,\alpha)$ as the triple $(S^nX, h^{(n)}, \alpha^{(n)})$, where $S^nX$ is
 the $n$th symmetric power $X^n/S_n$ of $X$, $h^{(n)}$ is the natural map $S^nX\to S^nX$
 induced by $h$ and $\alpha^{(n)}$ is defined in \cite[page 452]{RMC}. Then the pre-$\lambda$-ring
 structure on $\widehat{B}(G)$ is defined by the series
 $$
 \lambda_{[\widehat{X}]}(t)=1+[\widehat{X}]t+[S^2\widehat{X}]t^2+[S^3\widehat{X}]t^3+\ldots
 $$
\end{remark}

As an abelian group $\widehat{B}(G)$ is freely generated by the isomorphism classes of the
irreducible enhanced  $G$-sets $\widehat{X}_{H,k,\overline{h},\overline{\alpha}}$ described below.
Here $H$ is a subgroup of $G$, $k$ is a positive integer, $\overline{h}$ is a $G$-equivariant map
from $G/H$ to itself which can be identified with an element $\overline{h}\in N_G(H)/H$, where
$N_G(H)=\{a\in G: a^{-1}Ha=H\}$ is the normalizer of the subgroup $H$, and $\overline{\alpha}$ is
a 1-dimensional representation of the subgroup $H$ such that for $a\in H$ one has
\begin{equation}\label{alpha_condition}
 \overline{\alpha}(\overline{h}^{-1}a\overline{h})=\overline{\alpha}(a)\,.
\end{equation}
The enhanced  $G$-set $\widehat{X}_{H,k,\overline{h},\overline{\alpha}}$ is the triple
$(X, h, \alpha)$ defined by:
\begin{enumerate}
 \item[1)] the $G$-set $X$ is the Cartesian product $(G/H)\times \{0, 1, \ldots, k-1\}$
 of the irreducible $G$-set $G/H$ and the finite set $\{0, 1, \ldots, k-1\}$ with $k$ elements
 and with the trivial action of $G$;
 \item[2)] the $G$-equivariant map $h:X\to X$ is defined by
 $$
 h([a], i)=
 \begin{cases}
 ([a], i+1) &\text{for $i<k-1$,}\\
 ([a\overline{h}], 0) &\text{for $i=k-1$,}
 \end{cases}
 $$
 where $[a]$ is the class of $a\in G$ in the quotient $G/H$;
 \item[3)] $\alpha_{([a], i)}(b)=\overline{\alpha}(a^{-1}ba)$ for $b\in G_{([a], i)}=aHa^{-1}$.
 \end{enumerate}
The property (\ref{alpha_condition}) guarantees that the condition 3(b) in 
Definition~\ref{def1} is satisfied. Two irreducible enhanced $G$-sets
$\widehat{X}_{H,k,\overline{h},\overline{\alpha}}$ and $\widehat{X}_{H',k',\overline{h}',\overline{\alpha}'}$
are isomorphic if and only if $k=k'$ and there exists $g\in G$ such that $H'=gHg^{-1}$,
$\overline{h}'=g\overline{h}g^{-1}$, and $\overline{\alpha}'(b)=\overline{\alpha}(g^{-1}bg)$.

The ``equipped''  Burnside ring $\widetilde{B}(G)$ described above is generated by the conjugacy
classes of the pairs $(G/H, \alpha)$, where $H$ is a subgroup of $G$, $\alpha$ is a character of $H$.

For some constructions below we need a completion of the enhanced Burnside ring $\widehat{B}(G)$.

\begin{definition} \label{def3}
A {\em locally finite enhanced} $G$-{\em set} is a triple $(X, h, \alpha)$, where $X$ is a $G$-set,
$h$ and $\alpha$ are like in 
Definition~\ref{def1} such that the $h$-orbit
of each point $x\in X$ is finite and for each positive integer $m$ the set
$\{x\in X: h^m(x)=x\}$ is finite.
\end{definition}

Let $\widehat{\mathcal{B}}(G)$ be the Grothendieck ring of locally finite enhanced $G$-sets.
One can show that $\widehat{\mathcal{B}}(G)$ is the completion of the enhanced Burnside ring
$\widehat{B}(G)$ with respect to the following filtration. For $i\ge 0$, let
$$
F^i:=\{[(X, h, \alpha)]\in \widehat{B}(G): \forall x\in X\quad h^j(x)\neq x\quad\text{for }1\le j\le i\}\,.
$$
One can see that $F^i$ is an ideal in the ring $\widehat{B}(G)$ and one has the filtration
\begin{equation}\label{filtration}
\widehat{B}(G)=F^0\supset F^1\supset F^2\supset\ldots 
\end{equation}
One can show that the ring $\widehat{\mathcal{B}}(G)$ is the completion of the ring
$\widehat{B}(G)$ with respect to the filtration (\ref{filtration}).
Each element of $\widehat{\mathcal{B}}(G)$ can be written in a unique way as a series
$$
\sum a_{H,k,\overline{h},\overline{\alpha}} [\widehat{X}_{H,k,\overline{h},\overline{\alpha}}]
$$
with integer coefficients. (Elements from the enhanced Burnside ring $\widehat{B}(G)$ 
are represented by series with finitely many summands.)

\section{Enhanced equivariant Euler characteristics}\label{sec2}
For a topological $G$-space $Y$ (good enough, say, a subalgebraic variety) its
equivariant Euler characteristic $\chi^G(Y)$ with values in the Burnside ring $B(G)$
is defined (see, e.g., \cite{TtD}, \cite[Equation (2)]{GZ-FAA}; cf.\ \cite{Verdier}).
Here we define an enhanced version of the equivariant Euler characteristic
with values in the ring $\widehat{\mathcal{B}}(G)$ for a pair $(V, \varphi)$, where
$V$ is a complex manifold with a complex analytic action of the group $G$
and $\varphi:V\to V$ is a $G$-equivariant map with an additional property
(preservation of the ``age representation'': see below).
We assume that the manifold is from a class where all the objects below (say,
Lefschetz numbers) are defined. This is so, for example, for quasiprojective manifolds,
for compact complex manifolds with real boundaries and for the interiors of the latter ones.

Let $V$ be a complex manifold with a complex analytic action of a finite group $G$.
For a point $x\in V$ and an element $g$ of the isotropy subgroup $G_x$ of $x$,
its {\em age} (or fermion shift number) is defined in the following way
(\cite[Equation (3.17)]{Zaslow}, \cite[Subsection 2.1]{Ito-Reid}).
The element $g$ acts on the tangent space $T_xV$ as a complex linear operator
of finite order. It can be represented by a diagonal matrix with the diagonal
entries ${\mathbf{e}}[\omega_1]$, \dots, ${\mathbf{e}}[\omega_n]$, where $0\le\omega_i<1$
for $i=1, \ldots, n$ and ${\mathbf{e}}[r]:=\exp{(2\pi i r)}$ for a real number $r$.
The {\em age} of the element $g$ at the point $x$ is defined by 
${\rm age}_x(g)=\sum\limits_{i=1}^n\omega_i$. The function ${\rm age}_x: G_x\to \RR$
defines the one-dimensional representation $\alpha_x:G_x\to\CC^*$ by
$\alpha_x(g)={\mathbf{e}}[{\rm age}_x(g)]$.

For a $G$-equivariant map $\varphi$ of a $G$-manifold (not necessarily complex) into itself,
W.~L\"uck and J.~Rosenberg \cite{Luck} defined an equivariant Lefschetz number $L_V^G(\varphi)$
as an element of the Burnside ring $B(G)$. If the manifold $V$ is complex (and the $G$-action
is complex analytic), for a point $x$ of $V$ the one-dimensional representation $\alpha_x$
of the isotropy group $G_x$ is defined. Taking these representations into account, it is
possible, for a $G$-equivariant transformation $\varphi$ preserving the representations $\alpha_x$
($\alpha_{\varphi(x)}=\alpha_x$) to ``equip'' the L\"uck-Rosenberg Lefschetz number with
one-dimensional representations and thus to define its version $\widetilde{L}_V^G(\varphi)$
as an element of the ring $\widetilde{B}(G)$.

Let $(V,\varphi)$ be a pair consisting of a complex $G$-manifold $V$ and a $G$-equivariant
map $\varphi:V\to V$ such that $\alpha_{\varphi(x)}=\alpha_x$ for all $x\in V$.
For an enhanced $G$-set $\widehat{X}=(X,h,\alpha)$ representing an element of the group
$\widehat{\calB}(G)$, let $\widetilde{L}_X^G(h)\in\widetilde{B}(G)$ be the equipped Lefschetz number
corresponding to the transformation $h$ of $X$ considered as a zero-dimensional $G$-set
equipped with the representations $\alpha_x$, $x\in X$.

\begin{proposition}\label{prop1}
There exists a unique element $\widehat{X}$ of the completion $\widehat{\calB}(G)$ of the enhanced
Burnside ring $\widehat{B}(G)$ such that
\begin{equation} \label{equdef}
\widetilde{L}_X^{C_G(g)}(gh^m)= \widetilde{L}_V^{C_G(g)}(g\varphi^m)
\end{equation}
for all $m\ge 1$ and $g\in G$.
\end{proposition}

\begin{definition}
 The element $\widehat{X}=(X,h,\alpha)$ defined in Proposition~\ref{prop1} is called the
 {\em enhanced Euler characteristic} $\widehat{\chi}^G(V,\varphi)$ of the pair $(V,\varphi)$.
\end{definition}

\begin{remark}
 This definition is inspired by the definition of an equivariant zeta-function of a map from
 \cite{GZ-FAA}. In \cite{GZ-FAA}, there is a certain inaccuracy in the proof that the definition
 gives a well-defined object.
\end{remark}

\begin{proof*}{Proof of Proposition~\ref{prop1}}
For a manifold $V$ from the class under consideration the equivariant Lefschetz number
$\widetilde{L}^G_V(\varphi)$ of a $G$-equivariant map $\varphi: V \to V$ is an element
of the equipped Burnside ring $\widetilde{B}(G)$ represented by the (non-degenerate)
fixed point set of a generic $G$-equivariant deformation $\widetilde{\varphi}$ of $\varphi$
regarded as a (virtual) $G$-set consisting of points with multiplicities $\pm 1$.
Let us take all the fixed points of the map $g\widetilde{\varphi}^m$ for all $g \in G$ and
for all $m \leq k$. It is a finite (virtual) enhanced $G$-set, i.e.\ it represents an element
of $\widehat{B}(G)$. The limit of these sets for $k \to \infty$ yields an element of
$\widehat{\calB}(G)$ with the necessary properties.

Let 
$$
 [(X,h,\alpha)] = \sum a_{H,k,\overline{h},\overline{\alpha}} [\widehat{X}_{H,k,\overline{h},\overline{\alpha}}].
$$
The coefficient at $[(G/H, \overline{\alpha})]$ on the left hand side of Equation (\ref{equdef})
for $m=k$ and $g=\overline{h}^{-1}$ is a linear combination of the coefficients $a_\bullet$
including (only) a nonzero multiple of $a_{H,k,\overline{h},\overline{\alpha}}$ and the coefficients
$a_\bullet$ with smaller $k$. This gives a triangular system of linear equations which has a unique
solution (over the rational numbers; the representative described above guarantees that the
solution is over the integers).
\end{proof*}

\begin{remark}
One can see that an analogue of these arguments corrects the inaccuracy in the proof of
\cite[Proposition~1]{GZ-FAA}. From a formal point of view it is made here for manifolds.
One can, however, extend it to CW-complexes using, e.g., the fact that they can be embedded
into manifolds (``tubular neighbourhoods'').
\end{remark}

If the $G$-equivariant map $\varphi:V\to V$ is of finite order, one has the following equation for
$\widehat{\chi}^G(V,\varphi)$. For $H$, $k$, $\overline{h}$ and $\overline{\alpha}$ as above, let
$V^{(H, k, \overline{h}, \overline{\alpha})}$ be the set of points $x\in V$ such that $G_x=H$,
$\mathbf{e}[{\rm age}_x(g)]=\overline{\alpha}(g)$ for $g\in H$, $\varphi^i(x)\notin Gx$ for $1\le i\le k-1$,
$\varphi^k(x)=\overline{h}x$. Then
\begin{equation}\label{chi_finite_order}
 \widehat{\chi}^G(V,\varphi)=\sum \chi(V^{(H, k, \overline{h}, \overline{\alpha})}/N_G(H))\cdot
 [\widehat{X}_{H, k, \overline{h}, \overline{\alpha}}]\,.
\end{equation}

\section{Orbifold zeta functions}
Let $V$ be a complex manifold with a (complex analytic) action of the group $G$ and let
$\varphi:V\to V$ be a $G$-equivariant map such that
$\mathbf{e}[{\rm age}_{\varphi(x)}(g)]=\mathbf{e}[{\rm age}_{x}(g)]$
for $g\in G_x$. For an element $g\in G$ the map $\varphi$ preserves the fixed point set
$V^g=\{x\in V: gx=x\}$ of $g$ and is a $C_G(g)$-equivariant map on it. For a root of unity
$\beta\in\CC^*$, let $V^{g,\beta}=\{x\in V^g: \mathbf{e}[{\rm age}_x(g)]=\beta\}$.
The map $\varphi$ defines a map $\varphi_{g,\beta}:V^{g,\beta}/C_G(g)\to V^{g,\beta}/C_G(g)$.
The zeta function $\zeta_{\varphi_{g,\beta}}(t)$ of this map depends only on the conjugacy class of $g$.

For a series $\psi(t)\in 1+t\ZZ[[t]]$ (or for a rational function $\psi(t)$ with $\psi(0)=1$)
and for a root of unity $\beta\in\CC^*$, let
$\left(\psi(t)\right)_{\beta}:=\psi(\beta^{-1}t)$. Informally one can say that
$\left(\psi(t)\right)_{\beta}$ is obtained from $\psi(t)$ by multiplication
of all the roots and/or the poles by $\beta$, cf.\ the definition of $\left(\psi(t)\right)_{g}$
in \cite{PEMS}.

\begin{definition}
(cf.\ \cite[Definition~5.10]{ET}, \cite[Equatio~3]{PEMS})
 The {\em orbifold zeta function} of the map $\varphi$ is defined by
 \begin{equation}\label{zeta_continuous}
  \zeta_{G,V,\varphi}^{\rm orb}(t)=\prod_{[g]\in\text{Conj\,}G}\prod_{\beta}
  \left(\zeta_{\varphi_{g,\beta}}(t)\right)_{\beta}\,.
 \end{equation}
\end{definition}

Let $\widehat{X}=(X,h,\alpha)$ be a locally finite enhanced $G$-set (representing an element
of the ring $\widehat{\mathcal{B}}(G)$).
For $g\in G$ and a root of unity $\beta\in\CC^*$, let $X^{g,\beta}=\{x\in X: gx=x, \alpha_x(g)=\beta\}$
and let $h_{g,\beta}$ be the induced map from $X^{g,\beta}/C_G(g)$ to itself.
Let $\zeta_{h_{g,\beta}}(t)\in 1+t\ZZ[[t]]$ be the zeta function of the map $h_{g,\beta}$
considered as a map from the discrete topological space $X^{g,\beta}/C_G(g)$ to itself.
(This zeta function is defined (as a series) due to the condition of local finiteness of the
enhanced $G$-set $\widehat{X}$.)

\begin{definition}
 The {\em orbifold zeta function} of the element $[\widehat{X}]\in\widehat{\mathcal{B}}(G)$ is defined by
 \begin{equation}\label{zeta_discrete}
  \zeta_{[\widehat{X}]}^{\rm orb}(t)=\prod_{[g]\in\text{Conj\,}G}\prod_{\beta}
  \left(\zeta_{h_{g,\beta}}(t)\right)_{\beta}\,.
 \end{equation}
\end{definition}

One can see that $\zeta_{\bullet}^{orb}(t)$ is an ``additive to multiplicative homomorphism''
to $1+t\ZZ[[t]]$, i.e.\  $\zeta_{[\widehat{X}_1\coprod\widehat{X}_2]}^{orb}(t)=
\zeta_{[\widehat{X}_1]}^{orb}(t)\cdot\zeta_{[\widehat{X}_2]}^{orb}(t)$ for two enhanced
$G$-sets $\widehat{X}_1$ and $\widehat{X}_2$. Therefore this definition
can in an obvious way be extended to the Grothendieck ring $\widehat{\mathcal{B}}(G)$:
$$
\zeta_{[\widehat{X}_1]-[\widehat{X}_2]}^{\rm orb}(t)=
\zeta_{[\widehat{X}_1]}^{\rm orb}(t)/\zeta_{[\widehat{X}_2]}^{\rm orb}(t)\,.
$$

\begin{proposition}
 For a complex $G$-manifold $V$ and a $G$-equivariant map $\varphi:V\to V$ such that
 $\mathbf{e}[{\rm age}_{\varphi(x)}(g)]=\mathbf{e}[{\rm age}_{x}(g)]$ for $g\in G_x$, one has
 $$
 \zeta_{G,V,\varphi}^{\rm orb}(t)=\zeta_{\widehat{\chi}_G}^{\rm orb}(t)\,,
 $$
 where $\widehat{\chi}_G=\widehat{\chi}_G(V,\varphi)$.
\end{proposition}

\begin{proof}
To know $\zeta_{\varphi_{g,\beta}}(t)$, one needs to know the Lefschetz numbers of the iterates
of $\phi_{g,\beta}$ on the quotient of the part of the fixed point set of $\langle g \rangle$
with $\alpha_x(g)=\beta$ by the centralizer $C_G(g)$. To know them it is sufficient to determine
the fixed point sets of $a\varphi^k$ for all $k$ and $a \in C_G(g)$. They are determined by
the enhanced equivariant Euler characteristic $\widehat{\chi}_G$ and therefore the Lefschetz
numbers of the iterates of $\varphi_{g,\beta}$ (and thus the corresponding zeta functions)
for $(V,\varphi)$ and for an enhanced $G$-set $(X,h,\alpha)$ representing $\widehat{\chi}^G$
coincide. 
\end{proof}

Now let $G$ be abelian. In the sequel we will use a formula for the orbifold zeta function
of the basic element $[\widehat{X}_{H,k,\overline{h},\overline{\alpha}}]$ of the ring (group) $\widehat{B}(G)$.

\begin{lemma}\label{orbifold_basic}
 For $\widehat{X}=\widehat{X}_{H,k,\overline{h},\overline{\alpha}}$, one has
 $$
 \zeta^{\rm orb}_{[\widehat{X}]}(t)=\left(1-t^{{\rm lcm}(k,m)}\right)^{\frac{k\vert H\vert}{{\rm lcm}(k,m)}}\,,
 $$
 where $m=\frac{\vert H\vert}{\vert \{g\in H:\alpha(g)=1\}\vert}$\,, ${\rm lcm}(\cdot,\cdot)$ denotes
 the least common multiple.
\end{lemma}

\begin{proof}
 Recall that $\widehat{X}_{H,k,\overline{h},\overline{\alpha}}=(X,h,\alpha)$, where $X$, $h$ and $\alpha$
 are described above (Section~\ref{sec2}). One can see that $X^g=X$ for $g\in H$ and $X^g=\emptyset$
 otherwise. Moreover, $X^{g,\beta}=\emptyset$ for $\beta\neq\alpha(g)$ and $X^{g,\alpha(g)}=X$.
 Since $G$ is abelian, one has $C_G(g)=G$. The quotient $X/G$ consists of $k$ points permuted
 cyclically by the map $X/G\to X/G$ induced by $h$. The (usual) zeta function of this map is
 $1-t^k$. Therefore one has
 $$
 \zeta^{\rm orb}_{[\widehat{X}]}(t)=\prod_{g\in H}(1-t^k)_{\alpha(g)}=
 \prod_{g\in H}(1-(\alpha(g))^{-k}t^k)\,.
 $$
 The kernel ${\rm Ker\,}\alpha$ of $\alpha:H\to\CC^*$ is $\{g\in H: \alpha(g)=1\}$.
 Therefore its image consists of all the $m$th roots of unity with $m$ given by the equation above.
 The numbers $\left(\alpha(g)\right)^{-k}$, $g\in H$, are all the roots of unity of degree
 $\ell=\frac{m}{\gcd(m,k)}$, each one represented $\frac{\vert H\vert}{\ell}$ times.
 Therefore 
 $$
 \zeta^{\rm orb}_{[\widehat{X}]}(t)=
 \left(\prod_{j=0}^{\ell-1}
 \left(1-{\mathbf e}\left[\frac{j}{\ell}\right]t^k\right)\right)^{\frac{\vert H\vert}{\ell}}
 =\left(1-t^{k\ell}\right)^{\frac{\vert H\vert}{\ell}}=
 \left(1-t^{{\rm lcm}(k,m)}\right)^{\frac{k\vert H\vert}{{\rm lcm}(k,m)}}\,.
 $$
\end{proof}

\section{The enhanced equivariant Saito duality}
Let $\calG$ be a finite abelian group and let $\widehat{B}_1(\calG)$ be the subgroup of the enhanced
Burnside ring $\widehat{B}(\calG)$ generated by the isomorphism classes of finite enhanced sets
$(X,h,\alpha)$ such that $h(x)\in \calG x$ for all $x\in X$. (The subgroup $\widehat{B}_1(\calG)$
is in general not a subring.) As an abelian group it is freely generated by the isomorphism classes
of the enhanced $\calG$-sets $\widehat{X}_{H,1,\overline{h},\overline{\alpha}}$, where $\overline{h}\in \calG/H$,
$\overline{\alpha}\in\text{Hom}(H,\CC^*)$. In this case $h$ and $\alpha$ described in the definition
of the enhanced $\calG$-sets $\widehat{X}_{H,1,\overline{h},\overline{\alpha}}$ coincide with $\overline{h}$ and
$\overline{\alpha}$ respectively and therefore we shall use the notation $\widehat{X}_{H,1,h,\alpha}$.

Let $\calG^*=\text{Hom}(\calG,\CC^*)$  be the group of characters of $\calG$. (As an abstract group
it is isomorphic to $\calG$, but not in a canonical way.) For a subgroup $H\subset\calG$ its dual
subgroup $\widetilde{H}\subset\calG^*$ is defined by
$$
\widetilde{H}=\{\gamma\in\calG^*: \gamma(a)=1 \text{ for all } a\in H\} 
$$
(see \cite{BH2}).
The factor group $\calG/H$ is canonically isomorphic to $\widetilde{H}^*=\text{Hom}(\widetilde{H},\CC^*)$
and the group of characters $H^*=\text{Hom}(H,\CC^*)$ is canonically isomorphic to $\calG^*/\widetilde{H}$.
In this way the element $h\in \calG/H$ (i.e.\  a $\calG$-equivariant map $\calG/H\to\calG/H)$
defines a one-dimensional representation $\widetilde{h}$ of the subgroup $\widetilde{H}\subset\calG^*$
and the representation $\alpha$ of $H$ defines an element $\widetilde{\alpha}$ of the quotient
$\calG^*/\widetilde{H}$, i.e.\  a $\calG^*$-equivariant map $\calG^*/\widetilde{H}\to\calG^*/\widetilde{H}$.

\begin{definition}
 The {\em enhanced equivariant Saito duality} for the group $\calG$ is the group homomorphism
 $\widehat{D}_{\calG}:\widehat{B}_1(\calG)\to\widehat{B}_1(\calG^*)$ defined (on the generators) by
 $$
 \widehat{D}_{\calG}(\widehat{X}_{H,1,h,\alpha})=
 (\widehat{X}_{\widetilde{H},1,\widetilde{\alpha},\widetilde{h}})\,.
 $$
\end{definition}

One can easily see that $\widehat{D}_{\calG^*}\circ\widehat{D}_{\calG}={\rm id}$.

In the sequel we will need the following fact. Let $\alpha\in{\rm Hom}(H,\CC^*)$ be a
one-dimensional representation of a subgroup $H$ of $\calG$ with the kernel
${\rm Ker\,}\alpha=\{g\in H:\alpha(g)=1\}$ and let $\widetilde{\alpha}$ be the corresponding element
of $\calG^*/\widetilde{H}$. By abuse of notations let $\widetilde{\alpha}$ be also a representative
of $\widetilde{\alpha}$ in $\calG^*$.

\begin{lemma}\label{dual_to_Ker}
 The subgroup $\widetilde{{\rm Ker\,}\alpha}$ of $\calG^*$ dual to the kernel ${\rm Ker\,}\alpha$
 of the representation $\alpha$ is $\langle\widetilde{\alpha}\rangle+\widetilde{H}$, where
 $\langle\widetilde{\alpha}\rangle$ is the subgroup generated by $\widetilde{\alpha}$.
\end{lemma}

\begin{proof}
 Recall that the canonical pairing between a finite abelian group and the group
 of its characters (with values in the set of roots of unity) is denoted by
 $(\bullet,\bullet)_{\bullet}$. 
  We have to show that 
 $$
 \langle\widetilde{\alpha}\rangle+\widetilde{H}=\widetilde{{\rm Ker\,}\alpha}=
 \{\omega\in\calG^*:(\omega,g)_{\calG}=1\text{ for all }g\in{\rm Ker\,}\alpha\}\,.
 $$
 
 The element $\widetilde{\alpha}\in H^*=\calG^*/\widetilde{H}$
 corresponding to $\alpha$ is defined by $(\widetilde{\alpha},g)_H=\alpha(g)$ for all $g\in H$.
 Therefore $(\widetilde{\alpha},g)_H=1$ for $g\in{\rm Ker\,}\alpha$. Elements of $\widetilde{H}$
 pair with all the elements of $H$ (and thus of ${\rm Ker\,}\alpha\subset H$) with value $1$.
 Therefore $\langle\widetilde{\alpha}\rangle+\widetilde{H}\subset\widetilde{{\rm Ker\,}\alpha}$.
 
 On the other hand, let $\omega$ be an element of $\calG^*$ such that $(\omega, g)_{\calG}=1$ for all
 $g\in{\rm Ker\,}\alpha$. As a function on $H$, $(\omega, \bullet)_{\calG}$ is a power, say $\alpha^s$,
 of $\alpha$ (since the image of $\alpha$ is a cyclic subgroup of $\CC^*$). This means that
 $\omega=\widetilde{\alpha}^s \mod \widetilde{H}$.
\end{proof}

Let $G$ be a subgroup of $\calG$ and let $\widetilde{G}$ be the dual subgroup of $\calG^*$.
A $\calG$-set $X$ can be considered as a $G$-set, a
$\calG$-equivariant map $X\to X$ is at the same time $G$-equivariant, a one-dimensional representation
$\alpha_x$ of the isotropy subgroup $\calG_x$ defines (by restriction) a one-dimensional representation
of the isotropy subgroup $G_x=\calG_x\cap G$. This defines a natural (ring) ``reduction'' homomorphism
$R^{\calG}_{G}:\widehat{B}(\calG)\to\widehat{B}(G)$.

\begin{theorem}\label{theo1}
 For an element $\widehat{X}\in \widehat{B}_1(\calG)$, one has
 \begin{equation}\label{dual_orb}
 \zeta^{\rm orb}_{R^{\calG}_{G}[\widehat{X}]}(t)=
 \zeta^{\rm orb}_{R^{\calG^*}_{\widetilde{G}}\widehat{D}_{\calG}[\widehat{X}]}(t)\,.
\end{equation}
\end{theorem}

\begin{proof}
 We can verify (\ref{dual_orb}) for the generator
 $\widehat{X}=\widehat{X}^{\calG}_{H,1,h,\alpha}=(X,h,\alpha)$ of $\widehat{B}_1(\calG)$.
 The isotropy subgroup of any point of $X$ considered as a $G$-set is $G\cap H$.
 Let
 $$
 k:=\min\{\ell>0: h^{\ell}\in (G+H)/H\}=\frac{\vert\langle h\rangle+G+H\vert}{\vert G+H\vert}\,.
 $$
 One has $R^{\calG}_{G}[\widehat{X}]=
 \frac{\vert\calG\vert\cdot\vert G\cap H\vert}{k\vert H\vert\cdot\vert G\vert}
 [\widehat{X}^{G}_{H,k,h^k,\alpha_{\vert G\cap H}}]$.
 By Lemma~\ref{orbifold_basic} one has
 \begin{equation}\label{zeta1}
  \zeta^{\rm orb}_{R^{\calG}_{G}[\widehat{X}]}(t)=
  \left(1-t^{{\rm lcm}(k,m)}\right)^
  {\frac{\vert G\cap H\vert^2\vert\calG\vert}{{\rm lcm}(k,m)\cdot\vert G\vert\cdot\vert H\vert}}\,,
 \end{equation}
 where $m=\frac{\vert G\cap H\vert}{\vert{\rm Ker\,}\alpha\cap G\vert}$\,.
 One has $\widehat{D}_{\calG}[\widehat{X}]=\widehat{X}^{\calG^*}_{\widetilde{H},1,\widetilde{\alpha},\widetilde{h}}$.
 Therefore 
 \begin{equation}\label{zeta2}
  \zeta^{\rm orb}_{R^{\calG^*}_{\widetilde{G}}\widehat{D}_{\calG}[\widehat{X}]}(t)=
  \left(1-t^{{\rm lcm}(k',m')}\right)^
  {\frac
  {\vert \widetilde{G}\cap \widetilde{H}\vert^2\vert\calG^*\vert}
  {{\rm lcm}(k',m')\cdot\vert \widetilde{G}\vert\cdot\vert \widetilde{H}\vert}}\,,
  \end{equation}
 where $k'=\frac{\vert\langle \widetilde{\alpha}\rangle+\widetilde{G}+\widetilde{H}\vert}
 {\vert\widetilde{G}+\widetilde{H}\vert}$\,,
 $m'=\frac{\vert \widetilde{G}\cap \widetilde{H}\vert}
 {\vert{\rm Ker\,}\widetilde{h}\cap \widetilde{G}\vert}$\,.
 
 The subgroup of $\calG^*$ dual to $G\cap H\subset\calG$ (in the numerator of $m$) is
 $\widetilde{G}+\widetilde{H}$ (in the denominator of $k'$). Due to Lemma~\ref{dual_to_Ker}
 the subgroup of $\calG^*$ dual to ${\rm Ker\,}\alpha\cap G$ (in the denominator of $m$)
 is $\langle \widetilde{\alpha}\rangle+\widetilde{G}+\widetilde{H}$ (in the numerator of $k'$).
 Therefore $k'=m$. In the same way $m'=k$. Therefore the exponents ${\rm lcm}(k,m)$ and
 ${\rm lcm}(k',m')$ in (\ref{zeta1}) and (\ref{zeta2}) coincide. Using the fact that
 $\vert G\cap H\vert=\frac{\vert G\vert\cdot\vert H\vert}{\vert G+H\vert}$,
 it is easy to show that the other exponents in (\ref{zeta1}) and (\ref{zeta2}) coincide as well.
\end{proof}

\section{Enhanced Saito duality for invertible polynomials}
Let $f$ be an invertible polynomial in $n$ variables $x_1$, \dots, $x_n$. This means that $f$
is a quasihomogeneous polynomial of the form
$$
f(x_1, \ldots, x_n)=\sum_{i=1}^n a_i \prod_{j=1}^n x_j^{E_{ij}}\,,
$$
where $a_i$ are non-zero complex numbers, $E_{ij}$ are non-negative integers such that the matrix
$E=(E_{ij})$ is non-degenerate. The group of the (diagonal) symmetries of $f$ is
$$
G_f=\{\underline{\lambda}=(\lambda_1, \ldots, \lambda_n)\in(\CC^*)^n:
f(\lambda_1x_1, \ldots, \lambda_nx_n)=f(x_1, \ldots, x_n)\}\,.
$$
One can see that $G_f$ is an abelian group of order $\vert\det E\,\vert$. The action of the group $G_f$
on $\CC^n$ ($(\lambda_1, \ldots, \lambda_n)*(x_1, \ldots, x_n)=(\lambda_1x_1, \ldots, \lambda_nx_n)$)
defines the corresponding one dimensional representation $\alpha_f$ of $G_f$ on the exterior power
$\bigwedge^n\CC^n\cong \CC$: $\alpha_f(\lambda_1, \ldots, \lambda_n)=\lambda_1\cdot\ldots\cdot\lambda_n$
for $(\lambda_1, \ldots, \lambda_n)\in G_f$.

The Minor fibre $V_f=\{x\in\CC^n:f(x)=1\}$ of the (quasihomogeneous) polynomial $f$
is a complex manifold of dimension $n-1$ with the natural action of the group $G_f$. The polynomial
$f$ is quasihomogeneous with respect to the rational weights $q_1$, \dots, $q_n$ defined by the
equation
$$
E\cdot(q_1, \ldots, q_n)^T=(1,\ldots, 1)^T\,,
$$
i.e.\  $f({\rm e}[q_1\tau]x_1, \ldots, {\rm e}[q_n\tau]x_n)={\rm e}[\tau]f(x_1, \ldots, x_n)$.
The element $h_f=({\rm e}[q_1\tau], \ldots, {\rm e}[q_n\tau])\in(\CC^*)^n$ belongs to the group $G_f$
and represents the classical monodromy transformation of the polynomial $f$. It is also called
{\em the exponential grading operator}.

The Berglund-H\"ubsch transpose of $f$ is
$$
\widetilde{f}(x_1, \ldots, x_n)=\sum_{i=1}^n a_i \prod_{j=1}^n x_j^{E_{ji}}\,.
$$
The group $G_{\widetilde{f}}$ of symmetries of $\widetilde{f}$ is in a canonical way isomorphic
to the group $G_f^*={\rm Hom}(G_f,\CC^*)$ of characters of $G_f$ (see, e.g., \cite[Proposition 2]{BLMS}).
For a subgroup $G$ of $G_f$, one has the (Berglund-Henningson) dual subgroup $\widetilde{G}$
of $G_{\widetilde{f}}=G_f^*$ defined above. One has the following relation between the pair
$(h_f,\alpha_f)$ for the polynomial $f$ and the pair $(h_{\widetilde{f}},\alpha_{\widetilde{f}})$
for $\widetilde{f}$.

\begin{proposition}\label{dual}
 Under the identification $G_f^*=G_{\widetilde{f}}$, one has $\alpha_{\widetilde{f}}=h_f$,
 $h_{\widetilde{f}}=\alpha_f$.
\end{proposition}

\begin{proof}
 The canonical pairing between $G_{\widetilde{f}}$ and $G_f$ is defined by
 $$
 (\underline{\lambda},\underline{\mu})={\mathbf{e}}[(r_1,\ldots,r_n)E(s_1,\ldots, s_n)^T]\,,
 $$
 where $\underline{\lambda}=({\mathbf{e}}[r_1],\ldots,{\mathbf{e}}[r_n])\in G_{\widetilde{f}}$,
 $\underline{\mu}=({\mathbf{e}}[s_1],\ldots,{\mathbf{e}}[s_n])\in G_f$.
 The exponential grading operator $h_f=({\mathbf{e}}[q_1],\ldots,{\mathbf{e}}[q_n])\in G_f$ is defined
 by the equation $E(q_1,\ldots, q_n)^T=(1,\ldots, 1)^T$. Therefore, as a function on $G_{\widetilde{f}}$,
 it is equal to $(\underline{\lambda},h_f)={\mathbf{e}}[r_1+\ldots+r_n]=\lambda_1\cdot\ldots\cdot\lambda_n$,
 $\underline{\lambda}\in G_{\widetilde{f}}$, and thus coincides with $\alpha_{\widetilde{f}}$.
\end{proof}

Let $\widehat{\chi}^{G_f}(V_f,h_f)\in \widehat{\calB}(G_f)$ be the enhanced Euler characteristic
of the pair $(V_f,h_f)$. Since $h_f\in G_f$, one has $\widehat{\chi}^{G_f}(V_f,h_f)\in\widehat{B}_1(G_f)$.

The described situation is somewhat special. Namely, for an arbitrary finite enhanced $G_f$-set
$(X,h,\alpha)$ representing an element of $\widehat{B}_1(G_f)$, $\alpha_x$ is a representation of
the isotropy subgroup $(G_f)_x$ depending on $x$, $h$ can be considered as an element $h_x$ of
$G_f/(G_f)_x$ (also depending on $x$). Here we have a universal representation $\alpha_f$ of $G_f$
and also a distinguished element $h_f\in G_f$ so that $\alpha_x$ and $h_x$ are the corresponding
reductions of $\alpha_f$ and $h_f$ respectively. This permits to define the reduced versions
of the enhanced Euler characteristic and of the orbifold zeta function.

\begin{definition}
 The {\em reduced enhanced Euler characteristic} of $(V_f,h_f)$ is
 $$
 \overline{\widehat{\chi}}^{G_f}(V_f,h_f)=\widehat{\chi}^{G_f}(V_f,h_f)-[({\rm pt}, {\rm id}, \alpha_f)]\,.
 $$
\end{definition}

The reduced enhanced Euler characteristic of $(V_f,h_f)$ can be regarded (up to sign) 
as an enhanced equivariant Milnor number of the polynomial $f$.

Let $G$ be a subgroup of $G_f$.

\begin{definition} (cf.\ \cite[Definition 3]{PEMS})
 The {\em reduced orbifold zeta function} of the $G$-equivariant polynomial $f$ is
 $$
 \overline{\zeta}^{\rm orb}_{f,G}(t)=\zeta^{\rm orb}_{G,V_f,h_f}(t)\left/\prod_{g\in G}(1-t)_{\alpha_f(g)}\right.\,.
 $$
\end{definition}

The reduced orbifold zeta function $\overline{\zeta}^{\rm orb}_{f,G}(t)$ can be regarded
as the orbifold characteristic polynomial of the classical monodromy of $f$
(or the inverse to it depending on the number of variables).

\begin{remark}
 The reduced versions of the enhanced Euler characteristic and of the orbifold zeta function
 cannot be defined in the general setting since there is no distinguished one-dimensional
 representation.
\end{remark}

\begin{theorem}\label{theo2}
 One has
 $$
 \overline{\widehat{\chi}}^{G_{\widetilde{f}}}(V_{\widetilde{f}},h_{\widetilde{f}})=
 (-1)^n \widehat{D}_{G_f}\overline{\widehat{\chi}}^{G_f}(V_f,h_f)\,.
 $$
\end{theorem}

\begin{proof}
 For a subset $I\subset I_0=\{1, 2, \ldots, n\}$, let 
 $$
 (\CC^*)^I:= \{(x_1, \ldots, x_n)\in \CC^n: x_i\ne 0 {\rm \ for\ }i\in I, x_i=0 {\rm \ for\ }i\notin I\}
 $$
 be the corresponding coordinate torus.
 Let $\ZZ^n$ be the lattice of monomials in the variables $x_1$, \dots, $x_n$ (an $n$-tuple 
 $(k_1, \ldots, k_n)\in \ZZ^n$ corresponds to the monomial $x_1^{k_1}\cdot\ldots\cdot x_n^{k_n}$),
 let $\ZZ^I:=\{(k_1, \ldots, k_n)\in \ZZ^n: k_i=0 \mbox{ for }i\notin I\}$ and
 let $\mbox{supp\,} f\subset \ZZ^n$ be the set of monomials (with non-zero coefficients) in the
 polynomial $f$.
 
 The monodromy transformation $h_f$ respects the partitioning of $\CC^n$ into the tori $(\CC^*)^I$
 and therefore one has
 \begin{equation}\label{partitioning}
  \overline{\widehat{\chi}}^{G_f}(V_f,h_f)=\sum_{I\subset I_0}\chi_I^{G_f}(V_f, h_f)-
  [({\rm pt}, {\rm id}, \alpha_f)]\,,
 \end{equation}
 where $\chi_I^{G_f}(V_f, h_f)$ is the element of $\widehat{B}(G_f)$ represented by the triple
 $(\chi_I, h_f, \alpha_f)$, where $\chi_I$ is a representative of the equivariant Euler
 characteristic $\chi^{G_f}(V_f\cap(\CC^*)^I)\in B(G)$. In the proof of Theorem~1 in \cite{BLMS}
 it is shown that:
 \begin{enumerate}
  \item[1)] if $\vert{\rm supp\,} f \cap \ZZ^I\vert\neq\vert I\vert$, then
 $\vert{\rm supp\,} \widetilde{f} \cap \ZZ^{\overline{I}}\vert\neq\vert {\overline{I}}\vert$
 ($\overline{I}=I_0\setminus I$) and both $\chi_f^I$ and $\chi_{\widetilde{f}}^{\overline{I}}$
 are equal to zero;
 \item[2)] if $I$ is a proper subset of $I_0$ and $\vert{\rm supp\,} f \cap \ZZ^I\vert=\vert I\vert$,
 then $\vert{\rm supp\,} \widetilde{f} \cap \ZZ^{\overline{I}}\vert=\vert {\overline{I}}\vert$
 and the elements $\chi_f^I\in B(G_f)$ and $\chi_{\widetilde{f}}^{\overline{I}}\in B(G_{\widetilde{f}})$
 are (up to the sign $(-1)^n$) dual to each other in the sense of the equivariant (non-enhanced)
 Saito duality $D_{G_f}$;
 \item[3)] the element $\chi_f^{I_0}\in B(G_f)$ is (up to the sign $(-1)^n$) Saito dual to
 $(-1)\in B(G_{\widetilde{f}})$.
  \end{enumerate}
 Together with Proposition~\ref{dual} this implies the enhanced Saito duality (up to sign)
 between the summands in (\ref{partitioning}) and the corresponding summands in the same
 equation for $\widetilde{f}$.
\end{proof}

Theorems~\ref{theo1} and \ref{theo2} imply \cite[Theorem 1]{PEMS}.

\begin{corollary}
 One has 
 $$
 \overline{\zeta}^{\rm orb}_{\widetilde{f},G}(t)=\left(\overline{\zeta}^{\rm orb}_{f,G}(t)\right)^{(-1)^n}\,.
 $$
\end{corollary}

\bigskip
\noindent Leibniz Universit\"{a}t Hannover, Institut f\"{u}r Algebraische Geometrie,\\
Postfach 6009, D-30060 Hannover, Germany \\
E-mail: ebeling@math.uni-hannover.de\\

\medskip
\noindent Moscow State University, Faculty of Mechanics and Mathematics,\\
Moscow, GSP-1, 119991, Russia\\
E-mail: sabir@mccme.ru

\end{document}